\documentclass[11pt,a4paper]{amsart}
\usepackage[utf8]{inputenc}
\usepackage[T1]{fontenc}
\usepackage{amsfonts,amsmath,latexsym,amssymb} % these packages are required
\usepackage{amsthm}                % please use one of this two options for theorems
\usepackage{times}
\usepackage{xspace}
\usepackage{microtype}
\usepackage{paralist}
\usepackage[colorlinks, bookmarksopen=true, bookmarksnumbered=true, linkcolor=black,   urlcolor=black, citecolor=black]{hyperref}
\usepackage{lineno}
	
\renewcommand{\vec}[1]{ {\mathbf{#1} }  }
\newcommand{\im}{{\ensuremath{\mathrm{Im\,}}}}

\DeclareMathOperator{\Dom}{Dom}
\DeclareMathOperator{\dom}{Dom}
\DeclareMathOperator{\Ran}{\mathrm{Ran}}
\DeclareMathOperator{\Kern}{\mathrm{Ker}}

\DeclareMathOperator{\spec}{\mathrm{spec}}
\DeclareMathOperator{\trace}{\mathrm{tr}}

\DeclareMathOperator*{\wlim}{w-lim}
\newcommand{\linspan}{\mathrm{lin\ span}}
\newcommand{\fH}{\mathfrak{H}}
\newcommand{\fK}{\mathfrak{K}}

\newcommand{\cB}{\mathcal{B}}

\newcommand{\cL}{\mathcal{L}}

\newcommand{\ii}{\mathrm i}
\newcommand{\vA}{\vec{A}}
\newcommand{\vB}{\vec{B}}
\newcommand{\vV}{\vec{V}}

\newcommand{\eps}{\varepsilon}
\newcommand{\limeps}{\lim_{\eps \rightarrow 0^+} }
\newcommand{\normbig}[1]{\left\|#1\right\|}
\newcommand{\sca}[1]{\langle #1 \rangle}

\newcommand{\abs}[1]{\left| #1 \right|} % for absolute value
\newcommand{\NN}{\mathbb{N}}
\newcommand{\RR}{\mathbb{R}}
\newcommand{\CC}{\mathbb{C}}
\newcommand{\EE}{\mathsf{E}}
\newcommand{\ie}{\mbox{i.\,e.}\xspace}

{\bf}{\it}
{\bf}{\it}

\newtheorem{theorem}{Theorem}[section]

\newtheorem{lemma}[theorem]{Lemma}

\newtheorem{prop}[theorem]{Proposition}
\newtheorem{corollary}[theorem]{Corollary}
\newtheorem{hypo}[theorem]{Hypothesis}

\theoremstyle{definition}
\newtheorem{definition}[theorem]{Definition}
\newtheorem{example}[theorem]{Example}

\newtheorem{remark}[theorem]{Remark}

\title[Finite rank perturbations and solutions to the Riccati equation]{Finite rank perturbations and solutions to the operator Riccati equation}
\subjclass{Primary 47A62, 47A55; Secondary 47B15}
\keywords{Operator Riccati equation, singular spectrum, Herglotz functions}
\author[J.\ P.\ Gro{\ss}mann]{Julian P. Gro{\ss}mann}
\address{J. P.~Gro{\ss}mann,
Department Mathematik,
Friedrich-Alexander-Uni\-ver\-si\-t\"{a}t \linebreak[4] Er\-lang\-en-N\"{u}rnberg,
Cauerstr. 11,
D-91058 Erlangen,
Germany}
\email{jp-g@gmx.de}
\begin{document}

%\linenumbers

%% Title and other informations on the work
%\title[Finite rank perturbations and solutions to the Riccati equation]{Finite rank perturbations and solutions to the operator Riccati equation}
%\subjclass{Primary 47A62, 47A55; Secondary 47B15}
%\keywords{Operator Riccati equation, singular spectrum, Herglotz functions}
%\date{\today}
%
%%%% Author
%\author[J.\ P.\ Gro{\ss}mann]{Julian P. Gro{\ss}mann}
%\address{
%Department Mathematik\\
%Friedrich-Alexander-Uni\-ver\-si\-t\"{a}t Erlangen-N\"{u}rnberg\\
%Cauerstr. 11\\
%D-91058 Erlangen\\
%Germany \\
%\email{jp-g@gmx.de}}

\begin{abstract}
We consider an off-diagonal self-adjoint finite rank perturbation of a self-adjoint operator in a
complex separable Hilbert space $\fH_0 \oplus \fH_1$, where $\fH_1$ is finite dimensional. We describe the singular spectrum of the perturbed
operator and establish a connection with solutions to the operator Riccati equation.
In particular, we prove existence results for solutions in the case where the whole Hilbert space is finite dimensional.
\end{abstract}

\maketitle

\section{Introduction}
In the present article we analyse a special class of finite rank perturbations of a self-adjoint operator on a complex
separable Hilbert space $\fH$ and show how this is related to the existence of solutions
to the so-called \emph{operator Riccati equation}. This generalises results by Kostrykin and Makarov in \cite{Kostrykin.2005} where they considered rank one perturbations.

Let $\vA$ be a bounded self-adjoint operator on the Hilbert space $\fH$
	and $\fH_0 \subset \fH $  be a closed $\vA$-invariant subspace.
We choose $\fH_1 = \fH_0^{\perp} $ and define the self-adjoint operators
	$A_i := \vA|_{\fH_i} $ for $i =0,1$. 
	Assume that the perturbation $\vV: \fH \rightarrow \fH $
	is off-diagonal with respect to the orthogonal decomposition
	$ \fH = \fH_0 \oplus \fH_1$, \ie
	 $\Ran( \vV|_{\fH_0} ) \subset \fH_1 $ and
		$\Ran( \vV|_{\fH_1} ) \subset \fH_0 $.
	Consider then the
	perturbed
	self-adjoint operator
	\begin{equation*} \vB := \vA + \vV = \begin{pmatrix}
			A_0 & V \\ V^\ast & A_1
		\end{pmatrix}  
		~~\textnormal{with}~ \vV= \begin{pmatrix}
		0 & V \\ V^\ast & 0
		\end{pmatrix}
		\,,
		\end{equation*} 
	where $V : \fH_1 \rightarrow \fH_0$ is a bounded operator.
We will study the \emph{operator Riccati equation} associated with the operators above
		\begin{equation}
			A_1 X - X A_0 - X V X + V^\ast = 0 \label{eq:ric10}
			\,,
		\end{equation}
where the \emph{solution} $X$ is a densely defined operator from $\fH_0$ to $\fH_1$.
The name bears analogy to the familiar \emph{Riccati equation} as an ordinary differential equation and honours the Italian mathematician \emph{Jacopo Francesco Riccati} (1676 -- 1754).

It is well-known (see, e.g., \cite[Theorem 4.4]{Kostrykin.2003})
that the graph of a densely defined operator
$X: \fH_0 \supset \Dom(X) \rightarrow \fH_1$ is invariant for $\vB$
if and only if $X$ is a strong solution to the Riccati equation \eqref{eq:ric10}
in the sense of Definition \ref{def:strongric} below.

There are sufficient conditions which assure the existence of a
solution to the Riccati equation. 
If the spectra of the self-adjoint operators $A_0$ and $A_1$ are separated and the
operator norm of the perturbation $V$ is sufficiently small, then there is a bounded solution to \eqref{eq:ric10}. For details about the smallness of the perturbation see 
\cite[Theorem \nolinebreak 3.3]{Makarov.2013} in combination with \cite{Kostrykin.2003}.

If the spectra of $A_0$ and $A_1$ are even
subordinated, \ie
	\begin{equation*}
		\sup \spec(A_0) \leq \inf \spec(A_1) \, ,
	\end{equation*}
then a contractive solutions to the Riccati equation exists regardless of the norm of the perturbation $V$, see 
\cite{Kostrykin.2004}.
Similar results can be found in 
\cite{Adamyan.2001} and \cite{Kostrykin.2013}.

Note that in this work it is not assumed that the spectra
of the operators $A_0$ and $A_1$ are separated. Instead,
we require that the Hilbert space $\fH_1$
is finite dimensional. Under this assumption, we prove existence results for the Riccati equation.
We are mainly interested in bounded solutions, but we also prove
some statements about unbounded solutions.
In particular, all these results hold under the assumption that the whole Hilbert space
$\fH$ is finite dimensional.

Our main results are the following theorem and the deduced corollary.

\begin{theorem} \label{theo:main}
 Assume that $\fH_1$ is finite dimensional with $n := \dim \fH_1$, and suppose 
 that $\Ran V$ is a cyclic generating subspace
 for the operator $A_0$, \ie
 		\begin{equation*}
		\overline{ \linspan \Big\{ A_0^k v ~\big|~ k \in \NN_0, \, v \in \Ran V \Big\}} = \fH_0
		\,.
		\end{equation*}
 Then one has:
 \begin{enumerate}[(i)]
 		\item The multiplicity of the spectrum of $\vB$ is at most $n$. If there is an
 				eigenvalue of multiplicity $n$, then there is a bounded solution to the
 				Riccati equation \eqref{eq:ric10}.
 		\item \label{item:theo:main:(b)}
 		Assume that $\vB$ has at least $n$ eigenvalues
 		outside the point spectrum of $A_0$,
 		%$\spec_p(A_0)$, 
 		counted with multiplicities,
 		and let $U$ be the space spanned by the associated eigenvectors.
 		Furthermore, suppose that
 		\begin{equation*}
 			P_{\fH_1} U =  \fH_1 \,,
 		\end{equation*}
 	    where $P_{\fH_1} : \fH \rightarrow \fH$
 		is the orthogonal projection onto $\fH_1$.
 		Then the Riccati equation \eqref{eq:ric10} has a bounded solution.
 \end{enumerate}
\end{theorem}

\begin{corollary} \label{cor:main}
Let $\fH $ be finite dimensional
and assume that the spectra of $\vB $ and $A_0 $ are disjoint. 
Then there exists at least one bounded
solution to the Riccati equation \nolinebreak \eqref{eq:ric10}.
\end{corollary}

\section{Preliminaries}
In this section we want to fix some notations and explain facts about the concepts that we will use
in the following and need to prove Theorem \ref{theo:main}. 
Mainly, we present facts and proofs for readers that are not familiar with Herglotz functions, multiplicity of spectra and the decomposition of the spectrum into an absolutely continuous and
singular part.

We will
use the notation $\cL(\fH, \fK)$ for the set of bounded linear operators
from a Hilbert space $\fH$ to a Hilbert space $\fK$.
Moreover, we will write $\cL(\fH)$ instead of $\cL(\fH, \fH)$.
All considered Hilbert spaces are complex and separable. The spectrum
of a bounded linear operator $T: \fH \rightarrow \fH$ 
is denoted by $\spec(T)$ and the point spectrum, \ie the set of all
eigenvalues, is denoted by $\spec_p(T)$. Moreover, we will use the following notion of multiplicity of spectra, cf. \cite{Akhiezer.1977}.

\begin{definition}
 For a self-adjoint operator $T \in \cL(\fH)$ in a Hilbert space $\fH$, we call the minimal
 dimension of all cyclic generating subspaces the \emph{multiplicity of the spectrum of $T$}. 
Here, a subspace $U \subset \fH$ is called a
cyclic generating subspace for the operator $T$ if
 		\begin{equation*}
		 \linspan \Big\{ T^k u ~\big|~ k \in \NN_0, \, u \in U \Big\}
		\end{equation*}
is dense in $\fH$.
\end{definition}

With this definition the spectrum of an operator has multiplicity $1$ and is called \emph{simple}
if and only if there is a cyclic vector for this operator.
If we consider finite dimensional Hilbert spaces, the multiplicity of the spectrum above coincides with the maximal multiplicity of the eigenvalues of $T$. In infinite dimensional Hilbert spaces it is possible for the spectrum to have infinite multiplicity, e.g., the 
spectrum of the identity.

Now we will explain how so-called \emph{Herglotz functions} can be used to describe
self-adjoint operators and their spectra if the multiplicity is finite.
We will always write $\CC_+ := \{ z \in \CC ~|~ \im z > 0 \}$ for the upper complex half-plane
and
also use the following notion from \cite{Gesztesy.2000}: 

\begin{definition} \label{def:her}
\begin{enumerate}[(i)]
\item A holomorphic function $m : \CC_+ \rightarrow \CC$ is called
a \emph{scalar Herglotz function} if $\im m(z) \geq 0$ for all $z \in \CC_+$.
\item
An analytic function
$M: \CC_+ \rightarrow \CC^{n \times n}$ 
or $M: \CC_+ \rightarrow \cL(\CC^{n})$
with $n \in \NN$  is called a \emph{matrix-valued Herglotz function}
or \emph{Herglotz matrix} if 
	\begin{equation*} 
		\im M(z):= \frac{1}{2 \ii} (M(z) - M(z)^\ast) \geq 0
	\end{equation*} 
	for all $z \in \CC_+$.
\end{enumerate}
\end{definition}
A classical result in this theory is that every matrix-valued Herglotz function has a unique
integral representation, see \cite[Theorem 5.4]{Gesztesy.2000}. Therefore, there is a characteristic example of a Herglotz function if a matrix-valued measure $\Omega$ is given.
We call a map on the Borel sets of $\RR$, denoted by $\cB(\RR)$, with $\Omega : \cB(\RR) \rightarrow \CC^{n \times n}$  a
\emph{matrix-valued measure} if 
	\begin{equation*}
	\Omega_{y,x} : \cB(\RR) \rightarrow \CC \, 
	, \quad \Delta \mapsto \sca{y, \Omega(\Delta) x}_{\CC^n}
	\end{equation*}
is a (finite) complex measure for all $x,y \in \CC^n$. If we demand $\Omega(\Delta) \geq 0 $ for all Borel sets $\Delta \subset \RR$, the map $\Omega_{x,x}$ is a positive measure for all $x \in \CC^n$. 

\begin{example} \label{ex:Herglotz}
For each matrix-valued measure $\Omega$ with $\Omega(\Delta) \geq 0$ for all Borel sets $\Delta$ the map
	\begin{equation*}
	M : z \mapsto \int_{\RR} \frac{1}{t-z} \, d\Omega(t)
	\end{equation*}
defines a matrix-valued Herglotz function on $\CC_+$.
\end{example}

Lebesgue's decomposition theorem for ordinary positive measures that are 
$\sigma$-finite can easily
be generalised to complex measures, see \cite[Theorem 6.10]{Rudin.1987}, and to
matrix-valued measures. So for each matrix-valued measure $\Omega$ there is a unique decomposition into an absolutely continuous
and a singular measure with respect to the Lebesgue measure. 
By separating the atoms of the measure, the singular
part can additionally split into a singularly continuous part and a pure point part:
	\begin{equation*}
	\Omega = \Omega_{ac} + \Omega_s = \Omega_{ac} + \Omega_{sc} + \Omega_{pp}
	\,.
	\end{equation*}

With regard to the example of a Herglotz matrix above, we want to analyse these parts of the measure and it turns out that
an ordinary positive measure is sufficient for that task:

\begin{prop} \label{prop:Omega-equiv-omega}
Assume $\Omega : \cB(\RR) \rightarrow \CC^{n \times n}$ is a matrix-valued measure with $\Omega(\Delta) \geq 0$ for all Borel sets $\Delta \subset \RR$. Then
the positive measure $\omega(\Delta) := \trace( \Omega(\Delta) )$ is equivalent to $\Omega$, \ie
they have precisely the same null sets.
\end{prop}

\begin{proof}
 A null set for $\Omega$ is clearly a null set for $\omega$. Conversely, if we choose a
 Borel set $\Delta$ with $\omega(\Delta)= 0$, we can calculate
 	\begin{equation*}
 	2 \abs{ \Omega_{jk}(\Delta) } \leq \Omega_{jj}(\Delta) + \Omega_{kk}(\Delta) 
\leq 2 \trace (\Omega(\Delta))	
 	= 0 
 	\end{equation*}
 for all $1 \leq  j,k \leq n$. The first inequality is a standard property for non-negative matrices. See for example
 \cite[Lemma 5.1]{Gesztesy.2000}.
\end{proof}

In \cite{Gesztesy.2000} the authors give describing sets for the parts of the measure and call them supports. We also use this terminology here and call a Borel set $S \subset \RR$
a \emph{support} of a given Borel measure $\mu$, which can be positive, complex or matrix-valued, if
$\mu(\RR \setminus S) = 0$.
We call a support $S$ of $\mu$ \emph{minimal} if 
$ S \setminus T$ has Lebesgue measure zero
for any support $T \subset S$.

Since we will be merely interested in the singular part of the measure for analysing the Riccati equation, we just consider supports for the singular and the pure point part.

\begin{prop} \label{prop:support-M-Her}
 Let $\Omega : \cB(\RR) \rightarrow \CC^{n \times n}$ be a matrix-valued measure 
 that fulfils
  $\Omega(\Delta) \geq 0$ for all Borel sets $\Delta \subset \RR$ and
 $M: \CC_+ \rightarrow \CC^{n\times n}$ the matrix-valued Herglotz function from Example
 \ref{ex:Herglotz}. Then the set
 	\begin{equation*}
 	S_{\Omega, s} := \Big\{ \lambda \in \RR  ~\Big|~
							\lim_{\eps \rightarrow 0^+} \trace\, \im M(\lambda + \ii \eps) = \infty
							\Big\}
 	\end{equation*}
  is a minimal support of the singular part $\Omega_s$. The set
 	\begin{equation*} S_{\Omega, pp} := \Big\{ \lambda \in \RR  ~\Big|~
				\lim_{\eps \rightarrow 0^+} \eps  \trace  M(\lambda + \ii \eps)  \neq 0
							\Big\}
 	\end{equation*} 
 is the smallest support of the pure point part $\Omega_{pp}$.
\end{prop}

\begin{proof}
 By the equivalence of the measures $\Omega$ and $\trace \Omega$, one can use the support theorem
 \cite[Theorem 3.1]{Gesztesy.2000} for scalar Herglotz functions or the
 support theorem  \cite[Theorem 6.1]{Gesztesy.2000}  for Herglotz matrices.
\end{proof}

%	\begin{equation*}M(z) = C + Dz +  
%			 		\int_{\RR} \Big( \frac{1}{s -z} - \frac{s}{1+s^2} \Big) 
%					\, d \Omega(s)\label{eq:herglotzrepr3}
%	\end{equation*} 
%Here $\Omega: \cB(\RR) \rightarrow \CC^{n \times n} \cup \{ \infty \}$ is matrix-valued function
%defined on the Borel sigma algebra $\cB(\RR)$ with the property that for all bounded sets $B \subset \RR$ 
%	\begin{equation*}
%	\Omega_{y,x,B} : \cB(B) \rightarrow \CC \, 
%	, \quad E \mapsto \sca{y, \Omega(E) x}
%	\end{equation*}
%is a (finite) complex measure and with the property
%	\begin{equation*}
%	\sup \{ \abs{ \Omega_{y,x,B}(B) } ~|~ B \subset U \} = \abs{ \sca{y, \Omega(U) x} }
%	\end{equation*}
%for all unbounded sets $U \subset \RR$.

In the next proposition we present a fundamental example of a Herglotz function in relation to a self-adjoint operator
$T \in \cL(\fH)$, where we will write $\EE_{T}$ for its projector-valued spectral measure. In section \ref{sec:Eigen-and-singular} we will concretise this example.

\begin{prop} \label{prop:vsternv1}
 Let $T \in \cL(\fH) $ be self-adjoint with multiplicity of the spectrum $p \in \NN $
 and $n \geq p$ be a integer number.
 Moreover, let $V: \CC^n \rightarrow \fH $ be
 a linear operator such that $\Ran V $ is a cyclic generating
 subspace for $T$. Then
 	\begin{equation*} 
 	M: z \mapsto V^\ast (T - z)^{-1} V
 	\end{equation*} 
 is a matrix-valued Herglotz function for $z \in \CC_+ $ which can be represented by
 a matrix-valued measure $\Omega$ as
  	\begin{equation*} 
  	M(z) = \int_{\RR} \frac{1}{t-z} \, d\Omega(t) \,.
 	\end{equation*} 
  The measure
 $\Omega $ is equivalent to 
 the spectral measure $\EE_T $, \ie they have the same null sets. The set 
 	\begin{equation} S_{pp} := \Big\{ \lambda \in \RR ~\Big|~
				\lim_{\eps \rightarrow 0^+} \eps  \trace V^\ast (T - \lambda - i \eps)^{-1} V \neq 0
							\Big\} \label{eq:pp-spectrum-Herglotz}
 	\end{equation} 
 coincides with the point spectrum of $T$.
\end{prop}

\begin{proof}
Since $z \mapsto (T - z)^{-1}$ is analytic and because of
the first resolvent identity, one can write
  	\begin{equation*} 
  	 \im M(z) = \im z \left[ V^\ast (T -\overline{z})^{-1} (T - z )^{-1}
				V
				\right] \, .
	\end{equation*}
Obviously $\im M(z) \geq 0$ holds and therefore $M$ is a matrix-valued Herglotz function. If we now define
	\begin{equation*}
		\Omega(\Delta) := 
				V^\ast \EE_{T}(\Delta) V
	\end{equation*}
for every Borel set $\Delta \subset \RR $, we get a matrix-valued measure and the representation for $M$ holds by the spectral theorem. Clearly,  a null set of $ \EE_{T}$ is also 
a null set of
$\Omega$. On the other hand, a Borel set $\Delta$ with $\Omega(\Delta) = 0$ fulfils via the polarisation identity
	\begin{equation*}
	 \sca{v, \EE_T(\Delta^{\prime}) u} = 0 \,\text{ for all }\, u,v \in \Ran V 
	 \,, ~
%	 \Delta^{\prime} \subset \Delta 
%	 \,.
	\end{equation*}
and for all Borel sets $ \Delta^{\prime} \subset \Delta$.
Now we can use the spectral theorem for a measurable function $f(t) := t^m \chi_{\Delta}(t) t^k$,
where $\chi_{\Delta}$ is the characteristic function of $\Delta$ and $k,m$ non-negative integers:
	\begin{equation*}
	 0 = \int_{\Delta} t^{k+m} d \sca{v, \EE_T(t) u} = \sca{v, f(T) u} 
 		   = \sca{T^m v, \chi_{\Delta}(T) T^k u} \,.
	\end{equation*}
Since $u,v \in \Ran V$ and $\Ran V$ is a cyclic generating
 subspace for $T$, the following equation is true for all $x,y \in \fH$:
 	\begin{equation*}
 	0 =  \sca{y, \chi_{\Delta}(T) x} = \int_{\Delta} d \sca{y, \EE_T(t) x}
 			= \sca{y, \EE_T(\Delta) x}
 			\,.
 	\end{equation*}
That means that $\EE_T(\Delta) = 0$ and the measures are equivalent.

The smallest support of $\Omega_{pp}$ from Proposition \ref{prop:support-M-Her} is therefore also
a smallest support of the pure point part of the spectral measure $\EE_T$ and
it is well-known that the atoms of $\EE_T$ are exactly the eigenvalues of $T$.
\end{proof}

In the next definition we will decompose the spectrum 
of a self-adjoint operator in three parts. Analogously to above, one can generalise Lebesgue's
decomposition theorem even to a projector-valued measures like $\EE_{T}$. This is due to the fact that
	\begin{equation*}
		\Delta \mapsto \sca{y, \EE_{T}(\Delta) x}
	\end{equation*}
is a complex measure for every $x,y \in \fH$ which has a unique Lebegue decomposition.

\begin{definition}
 For a self-adjoint operator $T \in \cL(\fH)$ with spectral measure $\EE_T$ that has the Lebegue decomposition
 	\begin{equation*}
		\EE_T = \EE_{T, ac} + \EE_{T,s}  = \EE_{T,ac} + \EE_{T, sc} +\EE_{T,pp}
		\,,
	\end{equation*}
 we define the following sets for $w \in \{ ac, s, sc, pp \}$
  	\begin{equation*}
		\spec_{w}(T) := 
		\{ \lambda \in \RR ~|~ 
			\text{ every open neighbourhood } U \text{ of } \lambda 
			\text{ fulfils } \EE_{T, w}(U) \neq 0
		\}
		\,.
	\end{equation*}
 These closed sets 
 $\spec_{ac}(T)$,
  $\spec_{s}(T)$,
 $\spec_{sc}(T)$ and
 $\spec_{pp}(T)$ are called the \emph{absolutely continuous}, 
 \emph{singular},
 \emph{singularly continuous} and \emph{pure point spectrum of $T$}, respectively.
\end{definition}

Now it can be shown, cf. \cite[Chapter 10]{Kato.1966},
that for each self-adjoint operator $T$ there is a decomposition of its spectrum into
	 \begin{equation*} 
		\spec(T) = \spec_{ac}(T) \cup \spec_{sc}(T) \cup \spec_{pp}(T)
		\,.
	\end{equation*} 
Admittedly, none of this unions has to be disjoint. Note also that the pure point spectrum is in general larger than the set of eigenvalues $\spec_{p}(T)$ since the latter does not have to be closed. However,
	\begin{equation*} 
		\spec_{pp}(T) = \overline{\spec_{p}(T)}
	\end{equation*} 
always is true.

%%%%%%%%%%%%%%%%%%%%%%%%%%%%%%%%%%%%%%%%%%%%%%%%%%%%%%%%%%%%%%%%%
\section{Eigenvalues and singularly continuous spectrum of $\vB$} \label{sec:Eigen-and-singular}
Throughout this work we always assume the hypothesis below.

\begin{hypo} \label{hypo:problem}
	Let $\vB$ be a bounded self-adjoint operator which
	is represented with respect to the orthogonal decomposition $\fH = \fH_0 \oplus \fH_1$
	as an operator block matrix
		\begin{equation*} \vB := \begin{pmatrix}
			A_0 & V \\ V^\ast & A_1
		\end{pmatrix}  
		\,,
		\end{equation*} 
	where $A_j \in \cL(\fH_j)$ is self-adjoint for $j=0,1$ and $V \in \cL(\fH_1, \fH_0)$.
	
	Assume in addition that the Hilbert-space $\fH_1$ is finite dimensional
	and that $\Ran V$ is a cyclic generating subspace for the operator $A_0$, \ie
		\begin{equation*}
		 \linspan \Big\{ A_0^k v ~\big|~ k \in \NN_0, v \in \Ran V \Big\}
		\end{equation*}
	is dense in $\fH_0$.
\end{hypo}

Since this hypothesis claims that the multiplicity of the spectra $A_0$ and $A_1$ are not greater than 
$\dim \fH_1$, respectively,
the multiplicity of the spectrum of $\vB$ is also restrained.

\begin{lemma} \label{lemma:MultispectrumB}
 Assume Hypothesis \ref{hypo:problem}.
 Then $\fH_1 \subset \fH $ 
 is a cyclic generating subspace for $\vB$.
 In particular, the multiplicity of the 
 spectrum of $\vB$ cannot exceed
 $\dim \fH_1$.
\end{lemma}

\begin{proof}
 Set $n:=\dim \fH_1$
 and choose a basis $(e_i)_{i=1,\dots,n}$ of $\fH_1$. %such that $A_1$ is diagonal in this basis
 Since $\Ran V$ is a cyclic generating subspace for $A_0$ by Hypothesis \ref{hypo:problem},
 one concludes that
 	\begin{equation*}
 	\linspan 
 	\{ V e_i \oplus 0 \,,~ 0 \oplus e_i \in \fH_0 \oplus \fH_1
 	~|~
 									1 \leq i \leq n \}
 	\end{equation*}
 is a cyclic generating subspace for $\vB$. Obviously, if we substitute $ V e_i \oplus 0$
 with $ V e_i \oplus A_1 e_i$, the statement above will remain true.
 Since $\vB( 0 \oplus e_i ) = V e_i \oplus A_1 e_i $ holds, the space
 	 	\begin{equation*}
 	\linspan 
 	\{ 0 \oplus e_i \in \fH_0 \oplus \fH_1
 	~|~
 									1 \leq i \leq n \} = \fH_1
 	\end{equation*}
 is already a cyclic generating subspace for the operator $\vB$.
\end{proof}

The lemma above shows that the spectrum of the 
%self-adjoint 
operator $\vA := A_0 \oplus A_1$,
which could have the multiplicity $2 \cdot \dim \fH_1$, is always altered by the off-diagonal perturbation such that the multiplicity is at most only $\dim \fH_1$.

It is possible to classify
the eigenvalues of $\vB$ into three distinct
cases and it will turn out that this is necessary for finding solutions to the Riccati equation.
%Note that the Hilbert space $\fH_1$ is finite dimensional and
%therefore weak convergence is equivalent to operator norm convergence.

\begin{lemma} \label{lemma:eigenvaluesA}
 Assume Hypothesis \ref{hypo:problem}. 
 A real number $ \lambda \in \RR $ is an eigenvalue of 
 the operator $\vB$
 with multiplicity $k$ if and only if
 there is a set of $k$ linear independent vectors $\{ y_j \}_{j=1, \ldots, k} \subset \fH_1 $ with
 	\begin{equation*}
 		Vy_j \in \Ran(A_0 - \lambda) \, , \quad j =1, \dots, k \,,
 	\end{equation*}
 and for each $j$ one of the following statements holds:
	\begin{enumerate}[(i)]
		\item \label{item:a}   $ \lambda \notin \spec_p(A_0) $ and
			\begin{align*}
 				 (A_1 - \lambda) y_j = V^\ast (A_0 - \lambda)^{-1} V y_j
 			\,.
			\end{align*}
 		\item  \label{item:b}  $ \lambda \in \spec_p(A_0) $ and
 			\begin{align*}
 				 (A_1 - \lambda) y_j = 
 				\limeps V^\ast (A_0 - \lambda - \ii \eps)^{-1} V y_j
 				\,.
			\end{align*}
 		\item \label{item:lemma:worstcase}
 		 $ \lambda \in \spec_p(A_0) $ with an eigenvector $x\in \fH_0 $
 			and
 				\begin{align*}
 				 (A_1 - \lambda) y_j = 
 				\limeps V^\ast (A_0 - \lambda - \ii \eps)^{-1} V y_j  - V^\ast x
 				\,.
 				\end{align*}	
	\end{enumerate} 
\end{lemma}
Note that $\Ran(A_0 - \lambda) 
 				\subset \left( \Ran \EE_{A_0}(\{\lambda\}) \right)^{\perp}$
always holds and therefore the limit
$\limeps V^\ast (A_0 - \lambda - \ii \eps)^{-1} V y_j$ is well-defined by the spectral theorem. Here, $\EE_{A_0}$ stands for 
the spectral measure of the self-adjoint operator $A_0$.

\begin{proof}
  We will omit the proof of the multiplicity part because it is straightforward after having proved the following. 
  We will just prove
  here that a real number $\lambda$ is an eigenvalue of $\vB$ if and only if
  one of the three statements is fulfilled for a non-zero vector $y_1 \in \fH_1$ with
  $Vy_1 \in \Ran(A_0 - \lambda)$.
  We will start with the "only if" part.
  
  Note that a number $\lambda \in \RR$ is an eigenvalue of $\vB$ if and only if the two equations
	\begin{align}
	 &(A_0 - \lambda) y_0 = -V y_1 \label{eq:proofeigenwert1-A1} \\
	 &(A_1 - \lambda) y_1 = -V^\ast y_0   \label{eq:proofeigenwert1-A2}
	\end{align}
 are fulfilled for a non-zero vector $(y_0,y_1) \in \fH_0 \oplus \fH_1$.
  \vskip1ex
  \textit{First case:} If $\lambda \notin \spec_p(A_0)$, then $(A_0 - \lambda)$ is injective and
  we immediately get the equation of statement \eqref{item:a} for $y_1 \neq 0$.
  
  \textit{Second case:} If $\lambda \in \spec_p(A_0)$, the operator $(A_0 - \lambda)$ is not injective and therefore we change equation \eqref{eq:proofeigenwert1-A1}
  for an $\eps > 0$:
   \begin{equation}
   V^\ast (A_0 - \lambda -i \eps)^{-1} (A_0 - \lambda) y_0 = 
	- V^\ast (A_0 - \lambda - i \eps)^{-1} V y_1 \,.
	\label{eq:eigenvalues-proof1-eps}
   \end{equation}
  By the spectral theorem we can calculate
     \begin{equation*}
	\lim_{\eps \rightarrow 0^+} V^\ast (A_0 - \lambda -i \eps)^{-1} (A_0 - \lambda) y_0 = 
 				V^\ast (I_{\fH_0} - \EE_{A_0}(\{\lambda\}) ) y_0 \,.
   \end{equation*}
  For $y_0 \in \Ran \EE_{A_0}(\{\lambda\})^\perp$ we get the equation of \eqref{item:b} and
  for $y_0 \notin \Ran \EE_{A_0}(\{\lambda\})^\perp$ we get the equation of \eqref{item:lemma:worstcase}. In both cases $y_1 \neq 0$.
    \vskip1ex
    To show the "if" part of the claim above, one has to construct an eigenvector 
    $(y_0,y_1) \in \fH_0 \oplus \fH_1$ for $\vB$. Since $y_1$ with $V y_1 \in \Ran (A_0 - \lambda)$ is given,
    only $y_0$ is to construct. In the case \eqref{item:a} and \eqref{item:b} one can simply set
    	$y_0 \in \Ran \EE_{A_0}(\{\lambda\})^\perp$ such that
    		   \begin{equation*}
 				(A_0 - \lambda) y_0  = - Vy_1
   			\end{equation*}
   	holds. In the third case \eqref{item:lemma:worstcase} one has to do a similar reasoning and choose 
   	the vector
   	    	$y_0^{\prime} \in \Ran \EE_{A_0}(\{\lambda\})^\perp$ such that
    		   \begin{equation*}
 				(A_0 - \lambda) y_0^{\prime}  = - Vy_1
   			\end{equation*}
   	holds. Then just set $y_0 := y_0^{\prime} +x$.
\end{proof}

\begin{remark} \label{remark:lemmaforinfinite}
 The characterisation of the eigenvalues of $\vB$ in Lemma \ref{lemma:eigenvaluesA}
 remains true in the case of infinite dimensional $\fH_1$ if the (strong) limits are
 replaced with weak limits.
\end{remark}

\begin{example}
\label{ex:bad1}
We consider the Hilbert space $\fH = \fH_0 \oplus \fH_1 $ with $\fH_0 = \fH_1 = \CC^2$
and the linear operator $\vB: \fH \rightarrow \fH $ given by:
\begin{equation*}
\vB =   
\begin{pmatrix}
		A_0 & V \\ V^\ast & A_1
		\end{pmatrix}  
= 
\left(
\begin{array}{c|c}
		\begin{matrix}
				      ~ 1~ &   \\
				         & ~0 ~ 
		\end{matrix} & 
		\begin{matrix}
				      ~ 1 ~&   \\
				       ~1~  & ~1~  
		\end{matrix} \\ \hline
		\begin{matrix}
				       ~1~ & ~1~   \\
				         & ~1~  
		\end{matrix} 
	& 	\begin{matrix}
				       ~0 ~ &   \\
				         & ~0 ~  
		\end{matrix}
\end{array}
\right) \,.
\end{equation*}
There are three eigenvalues of $\vB$ that belong to condition \eqref{item:a} of
Lemma \ref{lemma:eigenvaluesA} and there is the eigenvalue $1$ that 
fulfils condition \eqref{item:lemma:worstcase}. By choosing 
$y_1 = (0,1)^T $ and $x = (-1, 0)^T $ we see that
	\begin{equation*}
	(A_1 - 1) y_1 = 
 				\limeps V^\ast (A_0 - 1 - i \eps)^{-1} V y_1 - V^\ast x
	\end{equation*}
holds.
\end{example}

The singular and singularly continuous spectrum of $\vB$ can be described 
by the use of minimal supports of the spectral measure which we will do in the following.
We write
$J_{\fH_1}: \fH_1 \rightarrow \fH $ for the inclusion map
and in this case the adjoint satisfies
$J_{\fH_1}^\ast(x) = P_{\fH_1}(x)$  for all $x \in \fH$.

%We write $\CC_+ := \{ z \in \CC ~|~ \im z > 0 \}$ for the upper complex half-plane. We
%also use the notion
%of Herglotz functions from \cite{Gesztesy.2000} and Definition \ref{def:her}.
\begin{prop}\label{prop:Mmatrix}
Assume Hypothesis \ref{hypo:problem}.
The map
$M: \CC_+ \rightarrow \cL(\fH_1)$ defined by
	\begin{equation*}
		M(z) := J_{\fH_1}^\ast
		 (\vB -z)^{-1} J_{\fH_1}
				\label{eq:Mzblock1} 
	\end{equation*}
is a matrix-valued Herglotz function with
		\begin{equation}
		M(z) = \big[ (A_1 -z) - V^\ast (A_0 -z)^{-1} V \big]^{-1} 
				\label{eq:Mzblock1-A} 		\, .
	\end{equation}
\end{prop}

\begin{proof}
Since $z \mapsto (\vB - z)^{-1}$ is analytic and because of
the first resolvent identity,
$M$ \nolinebreak is a matrix-valued Herglotz function, cf. \cite{Gesztesy.2000}. 
Note that also Example \ref{ex:Herglotz} is applicable to prove this.
The inverse of the Schur complement of
 $(\vB-z)$ shows equation \eqref{eq:Mzblock1-A}, see \cite[Proposition 1.6.2]{Tretter.2008}.
\end{proof}

The two propositions below explain a positive Borel measure which is equivalent to the spectral measure
of $\vB$
and describe the singularly continuous spectrum and the pure point spectrum of 
the perturbed operator $\vB$. This extends the results by Kostrykin and Makarov
in \cite{Kostrykin.2005}.

\begin{prop}\label{prop:Mmatrix-repre}
Assume Hypothesis \ref{hypo:problem}.
The Herglotz function
		\begin{equation*}
		m(z) = \trace \big( M(z)
		 \big)
				\label{eq:herglotz-m} 	
	\end{equation*}
admits the representation
	\begin{equation}
		m(z) = \int_{\RR} \frac{1}{t-z} \, d \omega(t)\, ,
				\label{eq:herglotz-m2} 	
				\,
	\end{equation}
where $\omega$ is a positive Borel measure with compact support.
Moreover, $\omega$ and the spectral measure
of $\vB$ are equivalent, \ie the null sets coincide.
\end{prop}

\begin{proof}
 From \cite[Theorem 5.4]{Gesztesy.2000} we know that $m$ is a scalar Herglotz function.
 	We define
	an operator-valued measure $\Omega $ with values in $\cL ( \fH_1)$ 
	by
		\begin{equation*}
		\Omega(\Delta) := 
				J_{\fH_1}^\ast \EE_{\vB}(\Delta) J_{\fH_1}
		\end{equation*}
	for every Borel set $\Delta \subset \RR $.
	%where $\EE_{\vB}$ denotes the spectral measure of $\vB$. 
	We easily see that
		\begin{equation*}
			\int \frac{d \Omega(t)}{t -z} = 
			J_{\fH_1}^\ast \int  \frac{d  \EE_{\vB}(t)}{t -z} \,  J_{\fH_1}
			= M(z) \quad \text{for all } z \in \CC_+
			\,.
		\end{equation*}
	Hence, $ \omega(\Delta) := \trace \Omega(\Delta) $ defines a positive measure 
	with compact support, which satisfies equation \eqref{eq:herglotz-m2}.
	
	Since by Lemma \ref{lemma:MultispectrumB} 
	the space $\fH_1$ is a cyclic generating subspace for $\vB$, the measure
	$\Omega$ is equivalent to $\EE_{\vB}$ by Proposition \ref{prop:vsternv1}. That $\omega$
	and $\Omega$ are equivalent has been shown in Proposition \ref{prop:Omega-equiv-omega}.
%	Naturally, a null set for $\EE_{\vB}$ is also a null set of $\omega$.
%	To see the converse, we consider a Borel set $\Delta$ with
%	$\omega(\Delta) = 0$. Choose a basis $(e_j)_{j=1, \dots, n}$ of $\fH_1$
%	with $n := \dim \fH_1$.
%	Then we observe that
%		\begin{equation*}
%			0 = \trace \Omega(\Delta) = \sum_{j=1}^n 
%			\sca{J_{\fH_1} e_j, \EE_{\vB}(\Delta) J_{\fH_1} e_j}
%			= \sum_{j=1}^n 
%			\sca{ e_j, \EE_{\vB}(\Delta) e_j} \,,
%		\end{equation*}
%	and each summand has to vanish.
%	Since by Lemma \ref{lemma:MultispectrumB} the space $\fH_1$ is cyclic for $\vB$,
%	this can only happen if $\EE_{\vB}(\Delta) = 0$.
\end{proof}

%Like in \cite{Gesztesy.2000} a Borel set $S \subset \RR$
%is called a support of a given Borel measure $\mu$ if
%$\mu(\RR \setminus S) = 0$.
%We call a support $S$ of $\mu$ minimal if 
%$ S \setminus T$ has Lebesgue measure zero
%for any support $T \subset S$.

%Recall that a set $S \subset \RR$ of a Borel measure $\mu$ is
%called a \emph{support of $\mu$} if $\mu(\RR \setminus S) = 0$. 
%Furthermore we call a support \emph{minimal} if every 

\begin{prop}\label{prop:Mmatrix-support}
Assume Hypothesis \ref{hypo:problem}.
The set
	\begin{equation*}
		S_s := \bigg\{ \lambda \in \RR ~\Big|~
				 \normbig{ \big[ (A_1 - \lambda- \ii \eps) -V^{\ast} 
				 (A_0 - \lambda - \ii \eps)^{-1} V \big]^{-1}} 
				 \xrightarrow{\eps \rightarrow 0^+} \infty 
			    \bigg\}
	\end{equation*}
is a minimal support of the singular part of the positive measure $\omega$
from Proposition \nolinebreak \ref{prop:Mmatrix-repre}. The set 
	\begin{align*}
		S_{pp} := \bigg\{ \lambda \in \RR ~\Big|~&
		\text{There is } 0 \neq y \in \fH_1 \text{ with } V y \in \Ran(A_0 - \lambda)
		\text{ and there is } \\
		& x \in \Ran \EE_{A_0}(\{ \lambda \}) \text{ such that } \\
				 & (A_1 - \lambda) y = 
				 \limeps V^{\ast} (A_0 - \lambda - \ii \eps)^{-1} V y - V^{\ast} x 
			    \bigg\}
	\end{align*}
is the set of all atoms of $\omega$. In particular, $S_{sc} := S_s \setminus S_{pp}$ is a minimal support for the singularly continuous part of $\omega$.
\end{prop}

\begin{proof}
 By \cite[Theorem 6.1]{Gesztesy.2000}, which is formulated in Proposition \ref{prop:support-M-Her},
 there is a minimal support of $\Omega_s$:
  	\begin{equation*}
 	S_{\Omega, s} := \Big\{ \lambda \in \RR  ~\Big|~
							\lim_{\eps \rightarrow 0^+} \trace\, \im M(\lambda + \ii \eps) = \infty
							\Big\} \,.
 	\end{equation*}
 Of course this is by the equivalence of the measures, see Proposition \ref{prop:Omega-equiv-omega}, also a minimal support for $\omega_s$. Obviously, $S_s \supset S_{\Omega, s}$ and therefore $S_s$ is a support of $\omega_s$ as well. It it is minimal by \cite[Theorem 5.4 (ii)]{Gesztesy.2000}.
 	
 The set $S_{pp}$ coincides with
 all eigenvalues of $\vB$. Note that we pushed the three cases of Lemma \ref{lemma:eigenvaluesA}
 into one formula here. By the equivalence of measures, $S_{pp}$  is the set of all atoms of $\omega$ and therefore the smallest support of $\omega_{pp}$.
\end{proof}

\begin{remark}
 The sets $S_s$ and $S_{pp}$ are connected to the spectrum of the perturbed operator $\vB$. We already noted that $S_{pp} = \spec_{p}(\vB)$ but the relation to $S_s$ is more subtle.
 In general neither $S_s \supset \spec_{s}(\vB)$ nor $S_s \subset \spec_{s}(\vB)$ is correct. However,
 $\overline{S_s} \supset \spec_{s}(\vB)$ is always true. Hence,
 if the singular spectrum of $\vB$ is non-empty, than $S_s$ is also non-empty.
\end{remark}

Now, we define subsets $K_{pp} \subset S_{pp}$ and $K_{sc} \subset S_{sc}$ of these supports of $\omega$, since not all points are suitable for the construction of a solution to the Riccati equation
as one can see in next section. The suitable subsets are given by:
	\begin{align*}
	 K_{pp} &:= \bigg\{ \lambda \in \RR ~ \Big|
 				\text{ There is } 0 \neq y \in \fH_1 \text{ with } \\
 				& \qquad \qquad (A_1 - \lambda) y = 
 				\limeps V^\ast (A_0 - \lambda - \ii \eps)^{-1} V y \\
 				&	\qquad\qquad	 ~ \text{ and }	\int \frac{1}{\abs{t - \lambda}^2} 
 				\, d \sca{Vy, \EE_{A_0}(t) Vy}	< \infty				\bigg\} \,
	\end{align*}
	and
	\begin{align*}
	K_{sc} &:= \bigg\{ \lambda \in \RR ~ \Big|
 				\text{ There is } 0 \neq y \in \fH_1 \text{ with } \\
 				& \qquad \qquad (A_1 - \lambda) y = 
 				\limeps V^\ast (A_0 - \lambda - \ii \eps)^{-1} V y \\
 				&	\qquad\qquad	 ~ \text{ and }	\int \frac{1}{\abs{t - \lambda}^2} 
 				\, d \sca{Vy, \EE_{A_0}(t) Vy}	= \infty				\bigg\} \,.
	\end{align*}
Note that $K_{pp} = S_{pp}$ if and only if there is no eigenvalue of $\vB$ which satisfies
the condition \textit{(\ref{item:lemma:worstcase})} of Lemma \ref{lemma:eigenvaluesA}. In particular, $K_{pp} = S_{pp}$ is fulfilled if
the point spectra of $A_0$ and $\vB$ are disjoint.

Furthermore, Kostrykin and Makarov have shown in \cite[Theorem 3.4]{Kostrykin.2005} that 
\linebreak
$K_{pp} =S_{pp}$
and $K_{sc} = S_{sc}$ hold if the Hilbert space $\fH_1$ is one-dimensional.
By using this result, they have constructed solutions to the Riccati equation for each $\lambda
\in S_s$ in the case that $\dim \fH_1 = 1$, see \cite[Theorem 4.3]{Kostrykin.2005}.
 In the following
section we extend their results about solutions to the Riccati equation for an arbitrarily
finite dimensional Hilbert space $\fH_1$.

\section{Solutions to the Riccati equation}

The operator Riccati equation \eqref{eq:ric10} a priori only makes sense
as an operator identity if
the solution $X$ is bounded and $\dom(X) = \fH_0$.
If one wants to include unbounded operators, a
generalised definition of solutions is required.
 We will use the same notion of a so-called strong solution as in \cite{Kostrykin.2005} and \cite{Kostrykin.2003}.

\begin{definition} \label{def:strongric}
	A densely defined, not necessarily bounded or closable, linear operator
		$X: \fH_0 \supset \Dom(X) \rightarrow \fH_1$ 
	is called a \textit{strong solution} to the Riccati equation \nolinebreak \eqref{eq:ric10} if
		\begin{equation*}
		\Ran(A_0 + V X)|_{\Dom(X)} \subset \Dom{(X)}
		\end{equation*}
	and
		\begin{equation*}
		A_1 X x - X (A_0 + V X) x  +  V^\ast x = 0 \quad \text{for all}~ x \in \Dom(X)
		\end{equation*}
	hold.
\end{definition}

\begin{hypo} \label{hypo:Lambda}
Assume Hypothesis \ref{hypo:problem}. Suppose that $K_{pp} \cup K_{sc}$ is not empty and that there are $n:= \dim \fH_1 $ linear independent
vectors $y_1, \dots, y_n \in \fH_1$ which satisfy
	\begin{equation} \label{eq:importantone}
		(A_1 - \lambda_k) y_k = 
 			\limeps V^\ast (A_0 - \lambda_k - \ii \eps)^{-1} V y_k \, 
 			,\qquad
 			\lambda_k \in K_{pp} \cup K_{sc} \,
	\end{equation}
for $k =1, \dots, n$.
Denote
$\Lambda := \{ (y_1, \lambda_1), \dots, (y_n, \lambda_n) \} $.
\end{hypo}

Under Hypthesis \ref{hypo:Lambda}, we define for $k=1, \dots, n$ the, not necessarily
orthogonal, projections $P_{\Lambda,k} : \fH_1 \rightarrow \fH_1$ by
  	\begin{align*}
	 &\Ran P_{\Lambda,k} = \linspan \{y_k\}~,~~ \\
	&\Kern P_{\Lambda,k} = 
		\linspan \{y_j ~|~ j \neq k \}
		\, .
	\end{align*}

We also define a possibly unbounded operator
$X_{\Lambda}: \fH_0 \supset  \dom(X_{\Lambda}) \rightarrow \fH_1$
on the domain
	\begin{equation*}\dom(X_{\Lambda}):= 
         \Bigg\{ x \in \fH_0  ~\bigg|~
			\lim_{\eps \rightarrow 0^+}
			\sum_{j=1}^n P_{\Lambda,j}^\ast V^\ast 
			(A_0 - \lambda_j + \ii \eps)^{-1} 
					x \in  \fH_1
				\Bigg\}
	\end{equation*}
by
	\begin{equation} 
		X_{\Lambda} x = \lim_{\eps \rightarrow 0^+}
		\sum_{j=1}^n P_{\Lambda,j}^\ast
				V^\ast (A_0 - \lambda_j + \ii \eps)^{-1}
			 x
			 \label{eq:XLAMBDA1}
	\, ,
	\end{equation}
which has the following properties.

\begin{prop} \label{prop:solutions}
 Assume Hypothesis \ref{hypo:Lambda} with a chosen $\Lambda$. Then:
 	\begin{enumerate}[(i)]
 		\item The linear operator $X_{\Lambda}$ is densely defined.
 		\item If $\lambda_j \in K_{sc}$ for at least one $j$, then
 		the operator $X_{\Lambda}$ is unbounded and non-closable.
 		\item If $\{ \lambda_1, \dots, \lambda_n \} \subset K_{pp} $, then the operator
 		$X_{\Lambda}$ is bounded.
 		\item  $A_0 x \in \dom(X_{\Lambda})$ for all 
 				$x \in \dom(X_{\Lambda})$.
 		\item $X_{\Lambda}$ is a strong solution to the Riccati equation \eqref{eq:ric10}.
 	\end{enumerate}
\end{prop}

\begin{proof}
 A proof of the statement (i) for the case $\dim \fH_1 = 1$
 can be found in \cite{Kostrykin.2005}. This
 proof has a straightforward generalisation to
 a finite dimensional $\fH_1$. With the same argument as in \cite[Lemma 4.1]{Kostrykin.2005} 
 one can show that the limit 
 	\begin{equation*}
 \lim_{\eps \rightarrow 0^+}
  P_{\Lambda,j}^\ast  V^\ast (A_0 - \lambda_j + i \eps)^{-1} \varphi
 \end{equation*}
  exists
 for $j \in \{1, \ldots, n\}$ and $\varphi \in \{p(A_0) u ~|~ p \text{ polynomial}, u \in \Ran V \}$.
 Since the latter set is dense in $\fH_0$, the operator $X_{\Lambda}$ is densely defined.
 
 To show (ii) we choose $\lambda_j \in K_{sc}$ and define for all $\eps \in (0,1]$
 the bounded operators $Y^{\eps} \in \cL(\fH_0, \fH_1)$ by
 	\begin{equation*}
 		Y^{\eps} x := P_{\Lambda,j}^\ast
 			V^\ast (A_0 - \lambda_j + i \eps)^{-1} x \,.
 	\end{equation*}
 A short calculation with the spectral theorem shows that the operator norm is given by
 	 	\begin{equation}
 	 \normbig{Y^{\eps}}_{\fH_0 \rightarrow \fH_1} = \abs{\alpha_j} 
 			\bigg( \int \frac{1}{ \abs{t - \lambda_j}^2 + \eps^2} \, 
 			d \sca{V y_{j}, \EE_{A_0}(t) V y_{j} } \bigg)^{1/2}
 				\label{eq:proofdimnbanachsteinhaus1}
 	\end{equation}
 where $\alpha_j \in \CC$ is a constant independent of $\eps$. 
 If $X_{\Lambda}$ was bounded, the operator defined by $Y := P_{\Lambda, j}^\ast X_{\Lambda}$ would also be bounded and therefore
   	 \begin{equation*}
	 	 \sup_{\eps \in (0,1]} \normbig{Y^{\eps} x}_{\fH_1} < \infty ~~\text{ for all } x \in \fH_0
 		\,.
 	\end{equation*}
%would be true.  
Since the 
uniform boundedness principle claims that $\sup_{\eps \in (0,1]} \normbig{Y^{\eps}}_{\fH_0 \rightarrow \fH_1}$ is finite and since that can be written by equation \eqref{eq:proofdimnbanachsteinhaus1} and the monotone convergence theorem
as
   	 \begin{equation*}
	 	 \int \frac{1}{ \abs{t - \lambda_j}^2} \, 
 			d \sca{V y_{\lambda_j}, \EE_{A_0}(t) V y_{\lambda_j} }  < \infty
	 	 \,,
 	\end{equation*}
there is a contradiction to $\lambda_j \in K_{sc}$.

 To show (iii), assume that $\{ \lambda_1, \dots, \lambda_n \} \subset K_{pp} $
 and define a bounded operator $Z: \fH_1 \rightarrow \fH_0$ by
 	\begin{equation*}
 		Z y := \wlim_{\eps \rightarrow 0^+} \sum_{j=1}^n 
					(A_0 - \lambda_j - \ii \eps)^{-1}  V P_{\Lambda,j} y ~,\quad y \in \fH_1
						\label{eq:adjointXLA1}\,.
 	\end{equation*}
 Since all $\lambda_j$ are eigenvalues of $\vB$ and $V P_{\Lambda,j} y \in \Ran(A_0 -\lambda_j) $
 for all $j =1, \dots, n$ and $y \in \fH_1$ by
 Lemma \ref{lemma:eigenvaluesA}, the weak limit is well-defined. 
 Choose $x \in \Dom(X_{\Lambda})$ and $y \in \fH_1$. Then
 	\begin{align*}
		\sca{x, Z y}_{\fH_0}
				&= \limeps \Big< x \, , \,   \sum_{j=1}^n 
					(A_0 - \lambda_j - \ii \eps)^{-1}  V P_{\Lambda,j} y\Big>_{\fH_0} \\
				&=  \limeps \Big< \sum_{j=1}^n  P_{\Lambda,j}^\ast
				   V^\ast (A_0 - \lambda_j + \ii \eps)^{-1}  x  \, , \,   
					 y\Big>_{\fH_1} 
				 =  \sca{ X_{\Lambda} x , y}_{\fH_1}\,,
 	\end{align*}
 so that $Z^\ast$ is an extension of $X_{\Lambda}$. Hence, $X_{\Lambda}$ is a
 closable operator of finite rank and therefore has to be bounded.
 
 Statement (iv) is shown by applying the spectral theorem. For each $j$ and all 
 $x \in \Dom(X_{\Lambda})$ one has
 	\begin{equation}
 	 \limeps P_{\Lambda,j}^\ast 
 	 V^\ast (A_0 - \lambda_j + \ii \eps)^{-1} (A_0 - \lambda_j) x = P_{\Lambda,j}^\ast V^\ast x 
 	 	\label{eq:limitfinite}
	\end{equation} 	 
 because $\fH_1$ is finite dimensional and 
 $\Ran( V P_{\Lambda,j} ) \subset \left( \Ran \EE_{A_0}(\{\lambda_j \}) \right)^{\perp} $.
 Therefore, we have 																
 $ A_0 x \in \Dom(X_{\Lambda})$ for all $x \in \Dom(X_{\Lambda})$.

 To show (v), we write the Riccati equation \eqref{eq:ric10} in the form
 		\begin{equation*}
 		\sum_{j=1}^n P_{\Lambda,j}^\ast(A_1 X - X A_0 - X V X + V^\ast) = 0 \,.
		\end{equation*}
 We choose $x \in \dom(X_{\Lambda})$ and calculate by using \eqref{eq:importantone}
 and \eqref{eq:XLAMBDA1}:
 \begin{align*}
		& P_{\Lambda,k}^\ast (A_1 X_{\Lambda} - X_{\Lambda} A_0 
		- X_{\Lambda} V X_{\Lambda})x \\[1ex]
		& 
		=
		P_{\Lambda,k}^\ast \Big( A_1 X_{\Lambda} x
		-  X_{\Lambda} A_0 x
			-\limeps \big(P_{\Lambda,k}\big)^\ast V^\ast 
						(A_0 - \lambda_k + \ii \eps)^{-1}V X_{\Lambda} x \Big)\displaybreak[0]\\[1ex]
		& =
		 P_{\Lambda,k}^\ast (A_1 -(A_1 - \lambda_k) ) X_{\Lambda}x
		-  P_{\Lambda,k}^\ast X_{\Lambda} A_0 x \displaybreak[0]\\[1ex]
		&= P_{\Lambda,k}^\ast X_{\Lambda} (\lambda_k - A_0)x \displaybreak[0]\\[1ex]
		&= \limeps P_{\Lambda,k}^\ast 	V^\ast 
						(A_0 - \lambda_k - \ii \eps)^{-1}	(\lambda_k - A_0) x\displaybreak[0]\\[1ex]
		&= -P_{\Lambda,k}^\ast  V^\ast x
		\,.
\end{align*}
In the last step we used equation \eqref{eq:limitfinite}.
\end{proof}

Finally, we are able to prove our main results:
\begin{proof}[Proof of Theorem \ref{theo:main}]
By Lemma \ref{lemma:MultispectrumB} the multiplicity of the spectrum of $\vB$ is
at most $n:= \dim \fH_1$. If there is an eigenvalue $\lambda$
with multiplicity $n$, then Lemma \nolinebreak \ref{lemma:eigenvaluesA}
shows that there are vectors $y_1, \dots, y_n \in \fH_1$ which span the Hilbert space $\fH_1$. 
Thus, also by Lemma \ref{lemma:eigenvaluesA} the inequality 
	\begin{equation*}
		\limeps \abs{ \trace V^\ast (A_0 - \lambda - \ii \eps)^{-1} V} < \infty
	\end{equation*}
holds
and one concludes that $\lambda \notin \spec_p(A_0)$. This is due to
Proposition \ref{prop:vsternv1}, in particular equation \eqref{eq:pp-spectrum-Herglotz}, and the fact that $\Ran V$ is a cyclic generating subspace for $A_0$.
Eventually, we construct a bounded solution $X_{\Lambda}$ to the Riccati equation with 
$\Lambda = \{ (y_{1}, \lambda), \dots, (y_{n}, \lambda) \} $ and Proposition  \ref{prop:solutions}.
This proves (i).

Statement (ii) is 
formulated in such a way that there exists at least one $\Lambda$ 
as in Hypothesis \ref{hypo:Lambda} such that 
Proposition \nolinebreak \ref{prop:solutions} is applicable.
\end{proof}

\begin{proof}[Proof of Corollary \ref{cor:main}]
    Since here it is not assumed that $\Ran V$ is a cyclic generating subspace for $A_0$,
    we define
		\begin{equation*}
		\mathfrak{K}_0 := \overline{\linspan \Big\{ A_0^k v ~\big|~ k \in \NN_0,
		\, v \in \Ran V \Big\}}
		\, ,
		\end{equation*}
	which is always a closed $A_0$-invariant subspace of $\fH_0$. One can choose $X|_{{\mathfrak{K}_0}^{\perp}} = 0$ for a solution $X$ to the Riccati equation \eqref{eq:ric10}, so that
	we can assume Hypothesis \ref{hypo:problem} without loss of generality.
	
	As $\fH$ is finite dimensional and spanned by the eigenvectors of $\vB$,
	we always find a bounded solution $X$ by
	Theorem \ref{theo:main} part \eqref{item:theo:main:(b)}.
\end{proof}

\section*{Acknowledgement}
This paper is based on the author’s Master’s thesis 
”Die Riccati-Gleichung unter Störungen endlichen Ranges”,
supervised by
Vadim Kostrykin at the University of Mainz.
The author would like to thank him for the introduction
to this field of research and many discussions about the Riccati equation.
Moreover, the author would also like to express his
gratitude to Albrecht Seelmann, Stephan Schmitz 
and Christoph Uebersohn for useful remarks.

%%%%%%%%%%%%%%%%%%%%%%%%%%%%%%%%%%%%%%%%%%%%%%%%%%%%%%%%%%%%
%%%%%%%%%%%%%%%%%%%%%%%%%%%%%%%%%%%%%%%%%%%%%%%%%%%%%%%%%%%%
%%% References
%%%%%%%%%%%%%%%%%%%%%%%%%%%%%%%%%%%%%%%%%%%%%%%%%%%%%%%%%%%%
%%%%%%%%%%%%%%%%%%%%%%%%%%%%%%%%%%%%%%%%%%%%%%%%%%%%%%%%%%%%

\end{document}